\documentclass[11pt,a4paper]{article}

\usepackage{inputenc}
\usepackage{amsmath}
\usepackage{bm}
\usepackage{bbold}
\usepackage{amsthm}

\usepackage{hyperref}

\title{A Tropical Extremal Problem with Nonlinear Objective Function and Linear Inequality Constraints\thanks{Advances in Computer Science: Proc. 6th WSEAS European Computing Conf. (ECC '12), WSEAS Press, 2012, pp.~146--151. (Recent Advances in Computer Engineering Series, Vol.~5)}
} 

\author{Nikolai Krivulin\thanks{Faculty of Mathematics and Mechanics, St.~Petersburg State University, 28 Universitetsky Ave., St.~Petersburg, 198504, Russia, 
nkk@math.spbu.ru.}
}

\date{}

\newtheorem{theorem}{Theorem}
\newtheorem{lemma}[theorem]{Lemma}
\newtheorem{corollary}[theorem]{Corollary}

\begin{document}

\maketitle

\begin{abstract}
We consider a multidimensional extremal problem formulated in terms of tropical mathematics. The problem is to minimize a nonlinear objective function, which is defined on a finite-dimensional semimodule over an idempotent semifield, subject to linear inequality constraints. An efficient solution approach is developed which reduces the problem to that of solving a linear inequality with an extended set of unknown variables. We use the approach to obtain a complete solution to the problem in a closed form under quite general assumptions. To illustrate the obtained results, a two-dimensional problem is examined and its numerical solution is given.
\\

\end{abstract}

\section{Introduction}

Models and methods of tropical (idempotent) mathematics \cite{Baccelli1993Synchronization,Cuninghame-Green1994Minimax,Kolokoltsov1997Idempotent,Golan2003Semirings,Heidergott2006Max-plus,Litvinov2007Themaslov,Butkovic2010Maxlinear} find expanding applications in many fields including engineering, computer science and operations research. There are a variety of real-life problems that can be represented and solved as tropical extremal problems. As related examples one can take idempotent algebra based solutions to scheduling \cite{Baccelli1993Synchronization,Heidergott2006Max-plus,Butkovic2009Introduction,Butkovic2009Onsome,Butkovic2010Maxlinear,Krivulin2012Algebraic} and location  \cite{Cuninghame-Green1994Minimax,Zimmermann2003Disjunctive,Tharwat2010Oneclass,Krivulin2011Algebraic,Krivulin2011Analgebraic,Krivulin2011Anextremal,Krivulin2012Anew} problems.

The tropical extremal problems are usually formulated as to minimize functionals, which are defined on finite-dimensional semimodules over idempotent semifields, subject to linear constraints in the form of equalities or inequalities. The problems under study may have linear objective functions as in  \cite{Zimmermann2003Disjunctive,Butkovic2009Introduction,Butkovic2010Maxlinear} or nonlinear function as in \cite{Krivulin2005Evaluation,Krivulin2006Solution,Krivulin2006Eigenvalues,Krivulin2009Methods,Krivulin2011Algebraic,Krivulin2011Analgebraic,Krivulin2011Anextremal,Gaubert2012Tropical,Krivulin2012Anew,Krivulin2012Solution}. In some cases both the objective function and the constraints in the problems appear to be nonlinear (see, e.g. \cite{Tharwat2010Oneclass,Butkovic2009Onsome}).

Many existing solution techniques \cite{Tharwat2010Oneclass,Butkovic2009Introduction,Butkovic2009Onsome,Butkovic2010Maxlinear,Gaubert2012Tropical} offer an iterative algorithm that gives particular solutions if any, and indicates that there is no solution otherwise. Other approaches \cite{Zimmermann2003Disjunctive,Krivulin2005Evaluation,Krivulin2006Eigenvalues,Krivulin2009Methods,Krivulin2011Algebraic,Krivulin2011Analgebraic,Krivulin2011Anextremal,Krivulin2012Anew,Krivulin2012Solution} allow one to get direct solutions in a closed form. However, in most cases only partial solutions are given leaving the problems without complete solution.

In this paper, we consider a multidimensional problem with a nonlinear objective function under linear inequality constraints. The problem is actually a constrained extension of the unconstrained problems in 
\cite{Krivulin2005Evaluation,Krivulin2006Eigenvalues,Krivulin2009Methods,Krivulin2011Algebraic,Krivulin2011Analgebraic,Krivulin2011Anextremal}. A solution approach is developed which reduces the problem to solving a linear inequality with an extended set of unknown variables. We use the approach to obtain a complete solution to the problem  in a closed form under quite general assumptions.

First we outline basic algebraic facts and present preliminary results that underlie subsequent developments. Together with an overview of known results including the solution of linear equations, a new binomial identity for matrix traces is presented. Furthermore, we formulate the extremal problem of interest and then derive a complete solution. To illustrate the obtained results, a two-dimensional problem is examined and its numerical solution is given.

\section{Preliminary Results}

The aim of this section is to provide a basis for further analysis and solution of tropical extremal problems. We briefly outline key algebraic facts and present preliminary results from \cite{Krivulin2006Solution,Krivulin2009Methods}. For further details and considerations, one can refer to \cite{Baccelli1993Synchronization,Cuninghame-Green1994Minimax,Kolokoltsov1997Idempotent,Golan2003Semirings,Heidergott2006Max-plus,Litvinov2007Themaslov,Butkovic2010Maxlinear}.

\subsection{Idempotent Semifield}

Let $\mathbb{X}$ be a set that is endowed with addition $\oplus$ and multiplication $\otimes$ and contains the zero element $\mathbb{0}$ and the identity $\mathbb{1}$. We suppose that both operations are associative and commutative, and multiplication is distributive over addition. Moreover, addition is idempotent, which means that $x\oplus x=x$ for all $x\in\mathbb{X}$, and multiplication is invertible; that is there exist an inverse $x^{-1}$ for any $x\in\mathbb{X}_{+}$, where $\mathbb{X}_{+}=\mathbb{X}\setminus\{\mathbb{0}\}$. Due to the above properties, $\langle\mathbb{X},\mathbb{0},\mathbb{1},\oplus,\otimes\rangle$ is usually referred to as the idempotent commutative semifield.

The power notation with integer exponents is routinely defined in the semifield. For any $x\in\mathbb{X}_{+}$ and an integer $p\geq1$, we have $x^{0}=\mathbb{1}$, $\mathbb{0}^{p}=\mathbb{0}$,
$$
x^{p}=x^{p-1}\otimes x=x\otimes x^{p-1},
\qquad
x^{-p}=(x^{-1})^{p}.
$$

Moreover, we assume that the rational powers are also defined which makes the semifield radicable.

In what follows, the multiplication sign $\otimes$ is omitted as is common in conventional algebra. The power notation is used only in the above sense.

Idempotent addition induces a partial order on $\mathbb{X}$ so that $x\leq y$ if and only if $x\oplus y=y$. We assume that the partial order can always be extended to a total order and so consider the semifield as linearly ordered. From here on, the relation symbols and the operator $\min$ are thought of in terms of this linear order.

An example of a commutative idempotent semifield that is radicable and linearly ordered is the real semifield  $\mathbb{R}_{\max,+}=\langle\mathbb{R}\cup\{-\infty\},-\infty,0,\max,+\rangle$.

\subsection{Idempotent Semimodule}

Consider the Cartesian product $\mathbb{X}^{n}$ with column vectors as its elements. For any vectors $\bm{x}=(x_{i})$ and $\bm{y}=(y_{i})$ from $\mathbb{X}^{n}$, and a scalar $c\in\mathbb{X}$, vector addition and scalar multiplication are routinely defined component-wise
$$
\{\bm{x}\oplus\bm{y}\}_{i}
=
x_{i}\oplus y_{i},
\qquad
\{c\bm{x}\}_{i}
=
cx_{i}.
$$

A geometric interpretation of the operations for the semimodule $\mathbb{R}_{\max,+}^{2}$ is given in the Cartesian coordinates in Fig.~\ref{F-VASM}.
\begin{figure}[ht]
\setlength{\unitlength}{1mm}
\begin{center}
\begin{picture}(35,45)

\put(0,5){\vector(1,0){35}}
\put(5,0){\vector(0,1){45}}

\put(5,5){\thicklines\vector(1,3){10}}
%\multiput(15,35)(0,-2.9){11}{\line(0,-1){2.0}}
\put(15,35){\line(0,-1){31}}

\put(5,5){\thicklines\vector(3,1){24}}
%\multiput(29,13)(-2.9,0){9}{\line(-1,0){1.9}}
\put(29,13){\line(-1,0){25}}

\put(5,5){\thicklines\line(4,5){24}}
\put(26,31.75){\thicklines\vector(1,1){3}}

%\multiput(29,35)(-2.9,0){9}{\line(-1,0){1.9}}
\put(29,35){\line(-1,0){25}}

%\multiput(29,35)(0,-2.9){11}{\line(0,-1){2.0}}
\put(29,35){\line(0,-1){31}}

\put(0,0){$0$}

\put(13,0){$y_{1}$}
\put(27,0){$x_{1}$}

\put(-1,13){$x_{2}$}
\put(-1,35){$y_{2}$}

\put(14,37){$\bm{x}$}

\put(31,14){$\bm{y}$}

\put(24,37){$\bm{x}\oplus\bm{y}$}

\end{picture}
\hspace{20\unitlength}
\begin{picture}(35,45)

\put(0,5){\vector(1,0){35}}
\put(5,0){\vector(0,1){45}}

\put(5,5){\thicklines\vector(1,3){5}}
\put(10,20){\line(0,-1){16}}
\put(10,20){\line(-1,0){6}}

\put(5,5){\thicklines\vector(2,3){20}}
\put(25,35){\line(0,-1){31}}
\put(25,35){\line(-1,0){21}}

\put(0,10){\line(1,1){30}}

\put(0,0){$0$}
\put(8,23){$\bm{x}$}
\put(27,32){$c\bm{x}$}

\put(-1,20){$x_{2}$}
\put(-2,35){$cx_{2}$}

\put(8,0){$x_{1}$}
\put(23,0){$cx_{1}$}

\end{picture}
\end{center}
%\vspace{-2ex}
\caption{Vector addition (left) and scalar multiplication (right) in $\mathbb{R}_{\max,+}^{2}$.}\label{F-VASM}
%\vspace{-1ex}
\end{figure}
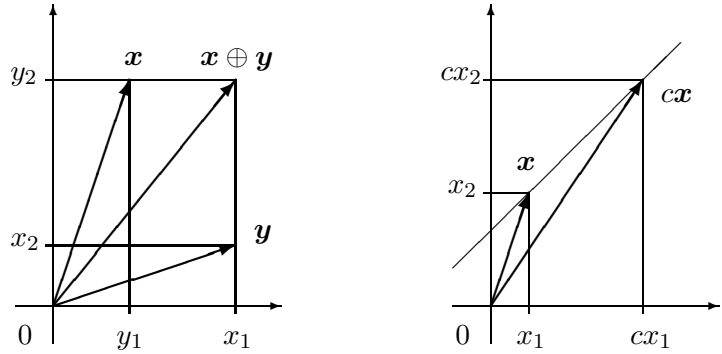

Equipped with these operations, the set $\mathbb{X}^{n}$ forms a semimodule over the idempotent semifield $\mathbb{X}$.

A vector with all zero elements is called the zero vector. A vector is regular if it has no zero elements. The set of all regular vectors in $\mathbb{X}^{n}$ is denoted by $\mathbb{X}_{+}^{n}$.

For any nonzero column vector $\bm{x}=(x_{i})$, we define a row vector $\bm{x}^{-}=(x_{i}^{-})$ with the elements $x_{i}^{-}=x_{i}^{-1}$ if $x_{i}>\mathbb{0}$, and $x_{i}^{-}=\mathbb{0}$ otherwise.

For any regular vectors $\bm{x}$ and $\bm{y}$, the component-wise inequality $\bm{x}\leq\bm{y}$ implies $\bm{x}^{-}\geq\bm{y}^{-}$.

A vector $\bm{y}$ is linearly dependent on vectors $\bm{x}_{1},\ldots,\bm{x}_{m}$ if $\bm{y}=c_{1}\bm{x}_{1}\oplus\cdots\oplus c_{m}\bm{x}_{m}$ for some scalars $c_{1},\ldots,c_{m}$. Specifically, vectors $\bm{y}$ and $\bm{x}$ are collinear if there exists a scalar $c$ such that $\bm{y}=c\bm{x}$.

Given vectors $\bm{x}_{1},\ldots,\bm{x}_{m}$, the set of linear combinations $c_{1}\bm{x}_{1}\oplus\cdots\oplus c_{m}\bm{x}_{m}$ for all $c_{1},\ldots,c_{m}$ is a linear span of the vectors, which forms a subsemimodule. A graphical example of the linear span of two vectors is shown in Fig.~\ref{F-LS}.
\begin{figure}[ht]
\setlength{\unitlength}{1mm}
\begin{center}
\begin{picture}(50,40)

\put(0,5){\vector(1,0){50}}
\put(5,0){\vector(0,1){40}}

\put(5,5){\thicklines\vector(1,3){4.25}}
\put(5,5){\thicklines\vector(2,3){17}}

\put(5,5){\thicklines\vector(4,1){23}}
\put(5,5){\thicklines\vector(2,1){34}}

\put(1.5,10){\thicklines\line(1,1){26}}
\multiput(2.5,11)(1,1){25}{\line(1,0){1}}

\put(17,0){\thicklines\line(1,1){26}}
\multiput(18,1)(1,1){25}{\line(-1,0){1}}

\put(39,30.5){\line(-1,0){35}}
\put(39,30.5){\line(0,-1){26.5}}

\put(10,18){\line(-1,0){6}}
\put(28,11){\line(0,-1){7}}

\put(5,5){\thicklines\vector(4,3){34}}

\put(0,0){$0$}

\put(7,22){$\bm{x}_{2}$}
\put(0,24){$c_{2}$}
\put(15,34){$c_{2}\bm{x}_{2}$}

\put(30,9){$\bm{x}_{1}$}
\put(32,0){$c_{1}$}
\put(42,20){$c_{1}\bm{x}_{1}$}

\put(31,34){$c_{1}\bm{x}_{1}\oplus c_{2}\bm{x}_{2}$}

\end{picture}
\end{center}
%\vspace{-2ex}
\caption{Linear span of two vectors in $\mathbb{R}_{\max,+}^{2}$.}\label{F-LS}
%\vspace{-1ex}
\end{figure}
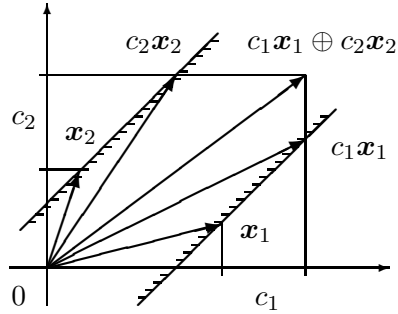

\subsection{Matrix Algebra}

Let $\mathbb{X}^{n\times n}$ be the set of square matrices of order $n$ with entries in $\mathbb{X}$. For any matrices $A=(a_{ij})$ and $B=(b_{ij})$ from  $\mathbb{X}^{n\times n}$, and a scalar $c\in\mathbb{X}$, matrix addition and multiplication together with scalar multiplication follow the conventional rules
$$
\{A\oplus B\}_{ij}
=
a_{ij}
\oplus
b_{ij},
\qquad
\{AB\}_{ij}
=
\bigoplus_{k=1}^{n}a_{ik}b_{kj},
\qquad
\{cA\}_{ij}
=
ca_{ij}.
$$

A matrix with all zero entries is called the zero matrix and denoted by $\mathbb{0}$. A matrix with all off-diagonal entries equal to $\mathbb{0}$ is a diagonal matrix. If all entries of a matrix above or below the diagonal are zero, the matrix is triangular. A diagonal matrix with $\mathbb{1}$ on the diagonal is the identity matrix denoted by $I$.

A matrix is reducible if it can be put into a block-triangular form by simultaneous permutations of rows and columns. Otherwise, the matrix is irreducible.

The matrix power is introduced in the regular way. For any matrix $A$ and integer $p\geq1$, we have
$$
A^{0}
=
I,
\qquad
A^{p}
=
A^{p-1}A
=
AA^{p-1}.
$$

\subsubsection{Matrix Trace}

For any matrix $A$, the trace of $A$ is routinely given by
$$
\mathop\mathrm{tr}A
=
a_{11}\oplus\cdots\oplus a_{nn}.
$$

It is not difficult to verify that for any matrices $A$ and $B$, and a scalar $c$ it holds
\begin{gather*}
\mathop\mathrm{tr}(A\oplus B)
=
\mathop\mathrm{tr}A
\oplus
\mathop\mathrm{tr}B,
\quad
\mathop\mathrm{tr}(AB)
=
\mathop\mathrm{tr}(BA),
\\
\mathop\mathrm{tr}(cA)
=
c\mathop\mathrm{tr}(A).
\end{gather*}

We apply the above properties to derive a binomial identity for traces to be used below.
\begin{lemma}
For any matrices $A$ and $B$, and an integer $m\geq1$ it holds that
\begin{equation}
\mathop\mathrm{tr}(A\oplus B)^{m}
=
\mathop\mathrm{tr}B^{m}
\\
\oplus\;
\bigoplus_{k=1}^{m}\mathop{\bigoplus\hspace{2.3em}}_{i_{1}+\cdots+i_{k}=m-k}\mathop\mathrm{tr}(AB^{i_{1}}\cdots AB^{i_{k}}).
\label{E-trABm}
\end{equation}
\end{lemma}
\begin{proof}
First we note that the power $(A\oplus B)^{m}$ can naturally be expanded as the sum of $2^{m}$ products formed by $k$ matrices $A$ and $m-k$ matrices $B$ taken in all possible orders for all $k=0,1,\ldots m$.

Furthermore, we take trace of the sum and consider those terms that have at least one matrix $A$ as a factor. Since two matrices commute under the trace operator, all these terms can be rearranged so that each term has $A$ as the first factor. It remains to see that any product with $k\geq1$ matrices $A$ is then represented as $(AB^{i_{1}})\cdots (AB^{i_{k}})$, where $i_{1},\ldots,i_{k}$ are nonnegative integers such that $i_{1}+\cdots+i_{k}=m-k$.
\end{proof}

\subsubsection{Spectral Radius}

As usual, a scalar $\lambda$ is an eigenvalue of a matrix $A$, if there exists a nonzero vector $\bm{x}$ such that
$$
A\bm{x}
=
\lambda\bm{x}.
$$

Every irreducible matrix has only one eigenvalue, whereas reducible matrices may have several eigenvalues. The maximal eigenvalue (in the sense of the linear order on $\mathbb{X}$) is called the spectral radius of $A$. It is directly calculated as
$$
\lambda
=
\bigoplus_{m=1}^{n}\mathop\mathrm{tr}\nolimits^{1/m}(A^{m}).
$$

The spectral radius $\lambda$ of a matrix $A$ offers a useful extremal property, which holds that
$$
\min\ \bm{x}^{-}A\bm{x}
=
\lambda,
$$
where the minimum is over all regular vectors $\bm{x}$.

\subsection{Linear Inequalities}

Let $A\in\mathbb{X}^{n\times n}$ be a given matrix, $\bm{x}\in\mathbb{X}^{n}$ be the unknown vector. Consider the problem of finding regular solutions $\bm{x}$ to the inequality
\begin{equation}
A\bm{x}
\leq
\bm{x}.
\label{I-Axx}
\end{equation}

For any matrix $A$, the solution of the inequality involves a function that is given by
$$
\mathop\mathrm{Tr}(A)
=
\mathop\mathrm{tr}A\oplus\cdots\oplus\mathop\mathrm{tr}A^{n},
$$
and a star operator that is defined as
$$
A^{\ast}
=
I\oplus A\oplus\cdots\oplus A^{n-1}.
$$

In the case of irreducible matrices, the solution is given by the following result.

\begin{theorem}\label{T-IAxx}
Let $\bm{x}$ be the general regular solution of inequality \eqref{I-Axx} with an irreducible matrix $A$.

Then the following statements hold:
\begin{enumerate}
\item If $\mathop\mathrm{Tr}(A)\leq\mathbb{1}$, then $\bm{x}=A^{\ast}\bm{u}$ for all $\bm{u}\in\mathbb{X}_{+}^{n}$.
\item If $\mathop\mathrm{Tr}(A)>\mathbb{1}$, then there is no regular solution.
\end{enumerate}
\end{theorem}

Fig.~\ref{F-GSIAxx} offers an example of solution of a linear inequality with a matrix $A=(\bm{a}_{1},\bm{a}_{2})$ in $\mathbb{R}_{\max,+}^{2}$.
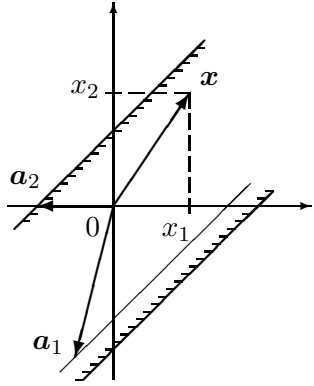
\begin{figure}[ht]
\setlength{\unitlength}{1mm}
\begin{center}
\begin{picture}(40,50)

\put(0,23){\vector(1,0){40}}
\put(14,0){\vector(0,1){50}}

\put(14,23){\thicklines\vector(-1,0){10}}
\put(1,20){\thicklines\line(1,1){25}}
\multiput(2,21)(1,1){24}{\line(1,0){1}}

\put(14,23){\thicklines\vector(-1,-4){5}}
\put(7,1){\line(1,1){25}}

\put(10,0){\thicklines\line(1,1){25}}
\multiput(10,0)(1,1){26}{\line(-1,0){1}}

\put(14,23){\thicklines\vector(2,3){10}}

\multiput(24,38)(0,-2.8){6}{\line(0,-1){2}}

\multiput(24,38)(-3,0){4}{\line(-1,0){2}}

\put(0,26){$\bm{a}_{2}$}
\put(3,4){$\bm{a}_{1}$}

\put(25,39){$\bm{x}$}

\put(20,19){$x_{1}$}

\put(8,38){$x_{2}$}

\put(10,19){$0$}

\end{picture}
\end{center}
%\vspace{-2ex}
\caption{Solution to a linear inequality in $\mathbb{R}_{\max,+}^{2}$.}\label{F-GSIAxx}
%\vspace{-1ex}
\end{figure}

\section{A Constrained Extremal Problem}

Given matrices $A,B\in\mathbb{X}^{n\times n}$, the problem is to find all regular solutions $\bm{x}\in\mathbb{X}_{+}^{n}$ so as to provide
\begin{equation}
\begin{split}
\min\ \bm{x}^{-}A\bm{x},
\\
B\bm{x}
\leq
\bm{x}.
\end{split}
\label{P-xAx}
\end{equation}

Below we give a general solution to the problem and then illustrate the solution with numerical examples in the framework of the semifield $\mathbb{R}_{\max,+}$.

\subsection{The Main Result}

A complete direct solution to the problem under quite general conditions is given as follows.

\begin{theorem}\label{T-xAxBxx}
Suppose that at least one of matrices $A$ and $B$ is irreducible, $\lambda>\mathbb{0}$ is a spectral radius of the matrix $A$, and $\mathop\mathrm{Tr}(B)\leq\mathbb{1}$.

Then the minimum in \eqref{P-xAx} is equal to
\begin{equation}
\theta
=
\bigoplus_{k=1}^{n}\mathop{\bigoplus\hspace{1.1em}}_{0\leq i_{1}+\cdots+i_{k}\leq n-k}\mathop\mathrm{tr}\nolimits^{1/k}(AB^{i_{1}}\cdots AB^{i_{k}})
\label{E-theta}
\end{equation}
and it is attained if and only if
\begin{equation}
\bm{x}
=
(\theta^{-1}A\oplus B)^{\ast}\bm{u}
\label{E-xtheta1ABu}
\end{equation}
for all regular vectors $\bm{u}$.
\end{theorem}
\begin{proof}
First we introduce an additional variable and then reduce problem \eqref{P-xAx} to an inequality. Existence conditions for solutions of the inequality is used to evaluate the variable, whereas the solution of the inequality is taken as the solution of the problem. 

Suppose $\theta$ is the minimum of the objective function in problem \eqref{P-xAx} and note that $\theta\geq\lambda>\mathbb{0}$.

The set of regular vectors $\bm{x}$ that yield the minimum is determined by the system
\begin{align*}
\bm{x}^{-}A\bm{x}
&=
\theta,
\\
B\bm{x}
&\leq
\bm{x}.
\end{align*}

Let us examine the first equality $\bm{x}^{-}A\bm{x}=\theta$. Since $\bm{x}^{-}A\bm{x}\geq\theta$ for all $\bm{x}\in\mathbb{X}_{+}^{n}$, the equality can be replaced by the inequality $\bm{x}^{-}A\bm{x}\leq\theta$.

Furthermore, we multiply the inequality by $\theta^{-1}\bm{x}$ from the left. Since $\bm{x}^{-}\bm{x}=\mathbb{1}$ and $\bm{x}\bm{x}^{-}\geq I$ for all regular $\bm{x}$, we have $\theta^{-1}A\bm{x}\leq\theta^{-1}\bm{x}\bm{x}^{-}A\bm{x}\leq\bm{x}$ and so get the inequality $\theta^{-1}A\bm{x}\leq\bm{x}$. Considering that left multiplication of the last inequality by $\theta\bm{x}^{-}$ gives the first one, both inequalities prove to be equivalent.

Now the solution set of the problem is given by the system of inequalities
\begin{align*}
\theta^{-1}A\bm{x}
&\leq
\bm{x},
\\
B\bm{x}
&\leq
\bm{x}.
\end{align*}

The above system is equivalent to one inequality
\begin{equation}
(\theta^{-1}A\oplus B)\bm{x}
\leq
\bm{x},
\label{I-Cthetaxx}
\end{equation}
which takes the form \eqref{I-Axx} with an irreducible matrix.

It follows from Theorem~\ref{T-IAxx} that inequality \eqref{I-Cthetaxx} has regular solutions if and only if
\begin{equation}
\mathop\mathrm{Tr}(\theta^{-1}A\oplus B)
\leq
\mathbb{1}.
\label{I-TrCtheta}
\end{equation}

Consider the function $\mathop\mathrm{Tr}(\theta^{-1}A\oplus B)$ and represent it in the form
$$
\mathop\mathrm{Tr}(\theta^{-1}A\oplus B)
=
\bigoplus_{m=1}^{n}\mathop\mathrm{tr}(\theta^{-1}A\oplus B)^{m}.
$$

By applying binomial identity \eqref{E-trABm} to the right-hand side, we have
$$
\mathop\mathrm{Tr}(\theta^{-1}A\oplus B)
=
\mathop\mathrm{Tr}B
\oplus
\bigoplus_{m=1}^{n}\bigoplus_{k=1}^{m}\mathop{\bigoplus\hspace{2.3em}}_{i_{1}+\cdots+i_{k}=m-k}\theta^{-k}\mathop\mathrm{tr}(AB^{i_{1}}\cdots AB^{i_{k}}).
$$

Furthermore, rearrangement of terms in the last sum gives the following expression
$$
\mathop\mathrm{Tr}(\theta^{-1}A\oplus B)
=
\mathop\mathrm{Tr}B
\oplus
\bigoplus_{k=1}^{n}\mathop{\bigoplus\hspace{1.2em}}_{0\leq i_{1}+\cdots+i_{k}\leq n-k}\theta^{-k}\mathop\mathrm{tr}(AB^{i_{1}}\cdots AB^{i_{k}}).
$$

Considering that $\mathop\mathrm{Tr}B\leq\mathbb{1}$ by the conditions of the theorem, inequality \eqref{I-TrCtheta} reduces to inequalities
$$
\mathop{\bigoplus\hspace{1.2em}}_{0\leq i_{1}+\cdots+i_{k}\leq n-k}\theta^{-k}\mathop\mathrm{tr}(AB^{i_{1}}\cdots AB^{i_{k}})
\leq
\mathbb{1},
$$
which must be valid for all $k=1,\ldots,n$.

By solving the inequalities with respect to $\theta$, we get inequalities
$$
\theta
\geq
\mathop{\bigoplus\hspace{1.2em}}_{0\leq i_{1}+\cdots+i_{k}\leq n-k}\mathop\mathrm{tr}\nolimits^{1/k}(AB^{i_{1}}\cdots AB^{i_{k}}),
$$
or equivalently, one inequality
$$
\theta
\geq
\bigoplus_{k=1}^{n}\mathop{\bigoplus\hspace{1.2em}}_{0\leq i_{1}+\cdots+i_{k}\leq n-k}\mathop\mathrm{tr}\nolimits^{1/k}(AB^{i_{1}}\cdots AB^{i_{k}}).
$$

In order for $\theta$ to be the minimum in problem \eqref{P-xAx}, the last inequality must be satisfied as an equality, which gives \eqref{E-theta}.

It remains to apply Theorem~\ref{T-IAxx} to inequality \eqref{I-Cthetaxx} so as to arrive at the solution in the form \eqref{E-xtheta1ABu}.
\end{proof}

As a particular case of the result, a solution to problem~\ref{P-xAx} without linear constraints can be derived.

\begin{corollary}\label{C-xAx}
Let, in addition to the assumptions of Theorem~\ref{T-xAxBxx}, $A$ be an irreducible matrix and $B=\mathbb{0}$.

Then the minimum in \eqref{P-xAx} is equal to $\lambda$ and it is attained if and only if
$$
\bm{x}
=
(\lambda^{-1}A)^{\ast}\bm{u}
$$
for all regular vectors $\bm{u}$.
\end{corollary}
\begin{proof}
It is sufficient to verify that we have $\theta=\lambda$. For this purpose we rewrite \eqref{E-theta} in the form
$$
\theta
=
\bigoplus_{k=1}^{n}\mathop\mathrm{tr}\nolimits^{1/k}(A^{k})
\oplus
\bigoplus_{k=1}^{n-1}\mathop{\bigoplus\hspace{1.1em}}_{1\leq i_{1}+\cdots+i_{k}\leq n-k}\mathop\mathrm{tr}\nolimits^{1/k}(AB^{i_{1}}\cdots AB^{i_{k}}).
$$

With $B=\mathbb{0}$ we arrive at the desired result. 
\end{proof}

Note that this result is consistent with that in \cite{Krivulin2012Solution}.

\subsection{Illustrative Examples}

Consider problem \eqref{P-xAx} in $\mathbb{R}_{\max,+}^{2}$ with matrices
$$
A
=
\left(
\begin{array}{rr}
0 & -3
\\
-5 & -2
\end{array}
\right),
\qquad
B
=
\left(
\begin{array}{rr}
0 & -8
\\
5 & -3
\end{array}
\right).
$$

First we apply Corollary~\ref{C-xAx} to solve the problem without constraints. We calculate the matrix
$$
A^{2}
=
\left(
\begin{array}{rr}
0 & -3
\\
-5 & -4
\end{array}
\right)
$$
and then find
$$
\lambda
=
\mathop\mathrm{tr}A
\oplus
\mathop\mathrm{tr}\nolimits^{1/2}(A^{2})
=
0
=
\mathbb{1}.
$$

Since $\lambda^{-1}A=A$, we have
$$
(\lambda^{-1}A)^{\ast}
=
A^{\ast}
=
I\oplus A
=
\left(
\begin{array}{rr}
0 & -3
\\
-5 & 0
\end{array}
\right).
$$ 

The solution to the problem takes the form
$$
\bm{x}
=
\left(
\begin{array}{rr}
0 & -3
\\
-5 & 0
\end{array}
\right)
\bm{u},
\qquad
\forall\bm{u}\in\mathbb{R}^{2}.
$$

Note that in a similar way, one can obtain the solution to the inequality $B\bm{x}\leq\bm{x}$ in the form
$$
\bm{x}
=
\left(
\begin{array}{rr}
0 & -8
\\
5 & 0
\end{array}
\right)
\bm{u},
\qquad
\forall\bm{u}\in\mathbb{R}^{2}.
$$

A graphical illustration of the solutions on a plane with the Cartesian coordinates is given in Fig.~\ref{F-UPIC}.
\begin{figure}[ht]
\setlength{\unitlength}{1mm}
\begin{center}
\begin{picture}(37,40)

\put(0,15){\vector(1,0){35}}
\put(18,0){\vector(0,1){40}}

\put(18,15){\thicklines\vector(-3,-2){6}}
\put(6,9){\thicklines\line(1,1){21}}
\multiput(7,10)(1,1){20}{\line(1,0){1}}

\put(7,6){\line(1,1){22}}

\put(18,15){\thicklines\vector(0,-1){10}}
\put(14,1){\thicklines\line(1,1){21}}
\multiput(15,2)(1,1){20}{\line(-1,0){1}}

\put(12,8){$\bm{a}_{2}$}

\put(20,3){$\bm{a}_{1}$}

\end{picture}
\hspace{5\unitlength}
\begin{picture}(37,40)

\put(0,15){\vector(1,0){35}}
\put(18,0){\vector(0,1){40}}

\put(18,15){\thicklines\vector(-4,-1){16}}
\put(0,13){\thicklines\line(1,1){23}}
\multiput(1,14)(1,1){22}{\line(1,0){1}}

\put(1,10){\line(1,1){24}}

\put(18,15){\thicklines\vector(0,1){10}}
\put(3,10){\thicklines\line(1,1){23}}
\multiput(4,11)(1,1){22}{\line(-1,0){1}}

\put(4,7){$\bm{b}_{2}$}

\put(20,22){$\bm{b}_{1}$}

\end{picture}
\end{center}
%\vspace{-2ex}
\caption{Solutions to an unconstrained problem (left) and inequality constraints (right) in $\mathbb{R}_{\max,+}^{2}$.}\label{F-UPIC}
%\vspace{-1ex}
\end{figure}
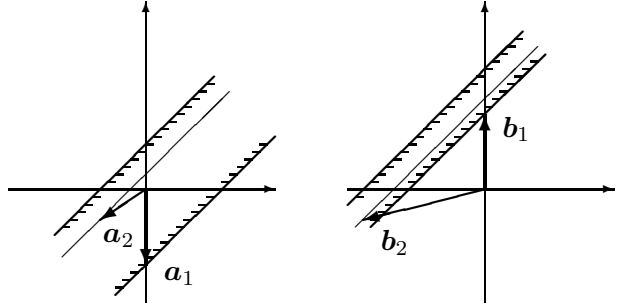

Now we use Theorem~\ref{T-xAxBxx} to handle the constrained problem. We get the matrix
$$
AB
=
\left(
\begin{array}{rr}
2 & -6
\\
3 & -5
\end{array}
\right)
$$
and then calculate
$$
\theta
=
\mathop\mathrm{tr}A
\oplus
\mathop\mathrm{tr}\nolimits^{1/2}(A^{2})
\oplus
\mathop\mathrm{tr}(AB)
=
2.
$$

Furthermore, we have
$$
\theta^{-1}A\oplus B
=
\left(
\begin{array}{rr}
0 & -5
\\
5 & -3
\end{array}
\right)
$$
and then arrive at
$$
(\theta^{-1}A\oplus B)^{\ast}
=
\left(
\begin{array}{rr}
0 & -5
\\
5 & 0
\end{array}
\right).
$$

Since the columns in the obtained matrix are collinear, we take only one of them, say the first. The solution set is then given by
$$
\bm{x}
=
\left(
\begin{array}{c}
0
\\
5
\end{array}
\right)
u
=
u\bm{x}_{0}
$$
for any nonzero number $u\in\mathbb{R}_{\max,+}$.

Fig.~\ref{F-CP} combines the solutions to the unconstrained problem and to the constraints with that of the constrained problem. The solution of the last problem is depicted with a double-thick line that is drawn through the end point of the vector $\bm{x}_{0}$ and coincides with a border of the feasible area defined by the constraints. 
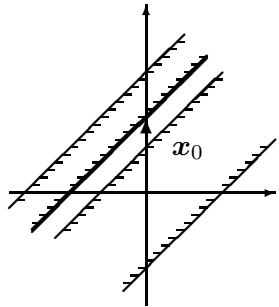
\begin{figure}[ht]
\setlength{\unitlength}{1mm}
\begin{center}
\begin{picture}(37,40)

\put(0,15){\vector(1,0){35}}
\put(18,0){\vector(0,1){40}}

%\put(18,15){\thicklines\vector(-3,-2){6}}
\put(6,9){\thicklines\line(1,1){22}}
\multiput(7,10)(1,1){21}{\line(1,0){1}}

%\put(9,8){\line(1,1){21}}

%\put(18,15){\thicklines\vector(0,-1){10}}
\put(15,2){\thicklines\line(1,1){20}}
\multiput(16,3)(1,1){19}{\line(-1,0){1}}

%\put(12,8){$\bm{a}_{2}$}

%\put(20,3){$\bm{a}_{1}$}

%\put(18,15){\thicklines\vector(-4,-1){16}}
\put(0,13){\thicklines\line(1,1){23}}
\multiput(1,14)(1,1){22}{\line(1,0){1}}

%\put(1,10){\line(1,1){23}}

\put(18,15){\thicklines\vector(0,1){10}}
\put(3,10){\thicklines\line(1,1){23}}
\multiput(4,11)(1,1){23}{\line(-1,0){1}}

\put(3,10.2){\thicklines\line(1,1){23}}
\put(3,9.8){\thicklines\line(1,1){23}}

%\put(4,7){$\bm{b}_{2}$}

%\put(20,22){$\bm{b}_{1}$}

\put(21,20){$\bm{x}_{0}$}

\end{picture}
\end{center}
%\vspace{-2ex}
\caption{Solution (double-thick line) to the problem under inequality constraints in $\mathbb{R}_{\max,+}^{2}$.}\label{F-CP}
%\vspace{-1ex}
\end{figure}

\bibliographystyle{utphys}

\bibliography{A_tropical_extremal_problem_with_nonlinear_objective_function_and_linear_inequality_constraints}

\providecommand{\href}[2]{#2}\begingroup\raggedright\begin{thebibliography}{10}

\bibitem{Baccelli1993Synchronization}
F.~L. Baccelli, G.~Cohen, G.~J. Olsder, and J.-P. Quadrat, {\em Synchronization
  and Linearity: An Algebra for Discrete Event Systems}.
\newblock Wiley Series in Probability and Statistics. Wiley, Chichester, 1993.
\newblock \url{http://www-rocq.inria.fr/metalau/cohen/documents/BCOQ-book.pdf}.

\bibitem{Cuninghame-Green1994Minimax}
R.~A. Cuninghame-Green,
  \href{http://dx.doi.org/10.1016/S1076-5670(08)70083-1}{``Minimax algebra and
  applications,''} in {\em Advances in Imaging and Electron Physics, Vol. 90},
  P.~W. Hawkes, ed., vol.~90 of {\em Advances in Imaging and Electron Physics},
  pp.~1--121.
\newblock Elsevier, 1994.

\bibitem{Kolokoltsov1997Idempotent}
V.~N. Kolokoltsov and V.~P. Maslov, {\em Idempotent Analysis and Its
  Applications}, vol.~401 of {\em Mathematics and Its Applications}.
\newblock Kluwer Academic Publishers, Dordrecht, 1997.

\bibitem{Golan2003Semirings}
J.~S. Golan, {\em Semirings and Affine Equations Over Them: Theory and
  Applications}, vol.~556 of {\em Mathematics and Its Applications}.
\newblock Springer, New York, 2003.

\bibitem{Heidergott2006Max-plus}
B.~Heidergott, G.~J. Olsder, and J.~van~der Woude, {\em Max-plus at Work:
  Modeling and Analysis of Synchronized Systems}.
\newblock Princeton Series in Applied Mathematics. Princeton University Press,
  Princeton, 2006.

\bibitem{Litvinov2007Themaslov}
G.~Litvinov, ``The maslov dequantization, idempotent and tropical mathematics:
  A brief introduction,''
  \href{http://dx.doi.org/10.1007/s10958-007-0450-5}{{\em J. Math. Sci.}
  {\bfseries 140} no.~3, (2007) 426--444},
  \href{http://arxiv.org/abs/0507014}{{\ttfamily arXiv:0507014 [math.GM]}}.

\bibitem{Butkovic2010Maxlinear}
P.~Butkovi\v{c}, \href{http://dx.doi.org/10.1007/978-1-84996-299-5}{{\em
  Max-linear Systems: Theory and Algorithms}}.
\newblock Springer Monographs in Mathematics. Springer, London, 2010.

\bibitem{Butkovic2009Introduction}
P.~Butkovi\v{c} and A.~Aminu, ``Introduction to max-linear programming,''
  \href{http://dx.doi.org/10.1093/imaman/dpn029}{{\em IMA J. Manag. Math.}
  {\bfseries 20} no.~3, (July, 2009) 233--249}.

\bibitem{Butkovic2009Onsome}
P.~Butkovi\v{c} and K.~P. Tam, ``On some properties of the image set of a
  max-linear mapping,'' {\em Contemp. Math.} {\bfseries 495} (2009) 115--126.

\bibitem{Krivulin2012Algebraic}
N.~Krivulin, ``Algebraic solutions to scheduling problems in project
  management,'' in {\em Recent Researches in Communications, Electronics,
  Signal Processing and Automatic Control}, pp.~161--166.
\newblock WSEAS Press, 2012.

\bibitem{Zimmermann2003Disjunctive}
K.~Zimmermann, ``Disjunctive optimization, max-separable problems and extremal
  algebras,'' \href{http://dx.doi.org/10.1016/S0304-3975(02)00231-1}{{\em
  Theoret. Comput. Sci.} {\bfseries 293} no.~1, (2003) 45--54}.

\bibitem{Tharwat2010Oneclass}
A.~Tharwat and K.~Zimmermann, ``One class of separable optimization problems:
  solution method, application,''
  \href{http://dx.doi.org/10.1080/02331930801954698}{{\em Optimization}
  {\bfseries 59} no.~5, (2010) 619--625}.

\bibitem{Krivulin2011Algebraic}
N.~Krivulin, ``Algebraic solutions to multidimensional minimax location
  problems with \uppercase{C}hebyshev distance,'' in {\em Recent Researches in
  Applied and Computational Mathematics}, pp.~157--162.
\newblock WSEAS Press, 2011.

\bibitem{Krivulin2011Analgebraic}
N.~Krivulin, ``An algebraic approach to multidimensional minimax location
  problems with chebyshev distance,'' {\em WSEAS Trans. Math.} {\bfseries 10}
  no.~6, (2011) 191--200.

\bibitem{Krivulin2011Anextremal}
N.~K. Krivulin, ``An extremal property of the eigenvalue for irreducible
  matrices in idempotent algebra and an algebraic solution to a
  \uppercase{R}awls location problem,''
  \href{http://dx.doi.org/10.3103/S1063454111040078}{{\em Vestnik St.
  Petersburg Univ. Math.} {\bfseries 44} no.~4, (2011) 272--281}.

\bibitem{Krivulin2012Anew}
N.~Krivulin, ``A new algebraic solution to multidimensional minimax location
  problems with \uppercase{C}hebyshev distance,'' {\em WSEAS Trans. Math.}
  {\bfseries 11} no.~7, (2012) 605--614.

\bibitem{Krivulin2005Evaluation}
N.~K. Krivulin, ``Evaluation of bounds on the mean rate of growth of the state
  vector of a linear dynamical stochastic system in idempotent algebra,'' {\em
  Vestnik St. Petersburg Univ. Math.} {\bfseries 38} no.~2, (June, 2005)
  42--51.

\bibitem{Krivulin2006Solution}
N.~K. Krivulin, ``Solution of generalized linear vector equations in idempotent
  algebra,'' {\em Vestnik St. Petersburg Univ. Math.} {\bfseries 39} no.~1,
  (March, 2006) 16--26.

\bibitem{Krivulin2006Eigenvalues}
N.~K. Krivulin, ``Eigenvalues and eigenvectors of matrices in idempotent
  algebra,'' {\em Vestnik St. Petersburg Univ. Math.} {\bfseries 39} no.~2,
  (2006) 72--83.

\bibitem{Krivulin2009Methods}
N.~K. Krivulin, {\em Methods of Idempotent Algebra in Problems of Complex
  Systems Modeling and Analysis}.
\newblock St.~Petersburg University Press, St.~Petersburg, 2009.
\newblock (in Russian).

\bibitem{Gaubert2012Tropical}
S.~Gaubert, R.~D. Katz, and S.~Sergeev, ``Tropical linear-fractional
  programming and parametric mean payoff games,''
  \href{http://dx.doi.org/10.1016/j.jsc.2011.12.049}{{\em J. Symbolic Comput.}
  {\bfseries 47} no.~12, (December, 2012) 1447--1478},
  \href{http://arxiv.org/abs/1101.3431}{{\ttfamily arXiv:1101.3431}}.

\bibitem{Krivulin2012Solution}
N.~Krivulin, ``Solution to an extremal problem in tropical mathematics,'' in
  {\em Tropical and Idempotent Mathematics: International Workshop}, G.~L.
  Litvinov, V.~P. Maslov, A.~G. Kushner, and S.~N. Sergeev, eds., pp.~132--139.
\newblock Moscow, 2012.

\end{thebibliography}\endgroup

\end{document}